\documentclass{article}
\usepackage{graphicx} 
\usepackage{amsmath,amssymb,amsthm,color,tikz}

\title{Smullyan's truth and provability}
\author{Taishi Kurahashi\footnote{Email: kurahashi@people.kobe-u.ac.jp}
\footnote{Graduate School of System Informatics, Kobe University, 1-1 Rokkodai, Nada, Kobe 657-8501, Japan.}
and Kohei Tominaga\footnote{Email: tominaga@stu.kobe-u.ac.jp}
\footnote{Graduate School of System Informatics, Kobe University, 1-1 Rokkodai, Nada, Kobe 657-8501, Japan.}}
\date{}

\theoremstyle{plain}
\newtheorem{thm}{Theorem}[section]
\newtheorem*{thm*}{Theorem}
\newtheorem{lem}[thm]{Lemma}
\newtheorem{prop}[thm]{Proposition}

\newtheorem*{fact*}{Fact}

\newtheorem*{prob*}{Problem}
\newtheorem*{cl*}{Claim}

\newtheorem*{scl*}{Subclaim}

\theoremstyle{definition}
\newtheorem{dfn}[thm]{Definition}

\newcommand{\Pred}{\mathtt{Pred}}
\newcommand{\Stc}{\mathtt{Sent}}

\newcommand{\True}{\mathtt{True}}
\newcommand{\False}{\mathtt{False}}
\newcommand{\FPT}{\textup{FPT}}

\newcommand{\FTarski}{\textup{F-Tarski}}
\newcommand{\TTarski}{\textup{T-Tarski}}
\newcommand{\SFTarski}{\textup{F-Tarski}^+}
\newcommand{\STTarski}{\textup{T-Tarski}^+}
\newcommand{\NGodel}{\textup{G1}}
\newcommand{\PGodel}{\textup{mG1}}
\newcommand{\SNGodel}{\textup{G1}^+}
\newcommand{\SPGodel}{\textup{mG1}^+}

\newcommand{\N}{\mathsf{n}}
\newcommand{\R}{\mathsf{r}}
\newcommand{\PA}{\mathsf{PA}}

\newcommand{\gn}[1]{\ulcorner#1\urcorner}
\newcommand{\num}[1]{\overline{#1}}

\begin{document}

\maketitle

\begin{abstract}
We revisit Smullyan's paper ``Truth and Provability'' (2013) for three purposes.
First, we introduce the notion of Smullyan models to give a precise definition for Smullyan's framework discussed in that paper. 
Second, we clarify the relationship between three theorems  proved by Smullyan and other newly introduced properties for Smullyan models in terms of both implications and non-implications.
Third, we construct two Smullyan models based on arithmetical ideas and show the correspondence between the properties of these Smullyan models and those concerning truth and provability in arithmetic.
\end{abstract}

\section{Introduction}

G\"odel's First Incompleteness Theorem and Tarski's Undefinability Theorem are major achievements in mathematical logic and have had a great impact on mathematics and other fields.
A version of the First Incompleteness Theorem states that every computable sound extension of Peano arithmetic is incomplete. 
The Undefinability Theorem states that the set of all true sentences in the standard model of arithmetic is not definable in the standard model. 
These theorems are positioned nowadays as basic results in mathematical logic, but of course understanding the proofs of these theorems requires a reasonable amount of knowledge and experience in mathematical logic. 

The structures of the proofs of these theorems are themselves very interesting, and Smullyan wrote a number of books and papers to bringing the essence of these structures to the general reader (e.g.~\cite{Smu78,Smu82,Smu88,Smu09,Smu13b,Smu13}). 
In particular, the article \cite{Smu13} titled ``Truth and Provability'' published in \textit{The Mathematical Intelligencer} (2013) provided a concise presentation of the structures of the proofs of these theorems using a very simplified framework dealing with finite strings of symbols. 
He wrote as follows: 
\begin{quotation}
{\small 
The purpose of this article is to provide the general reader, even those readers with no familiarity with the symbolism of mathematical logic, with the essential ideas behind the proofs of the G\"odel and Tarski theorems. 
We do this by constructing a very simple system (an abstraction of part of Reference [1]\footnote{Reference \cite{Smu82} of the present article.}), which, despite its simplicity, has enough power for the Tarski and G\"odel arguments to go through. 
First we address Tarski’s theorem, and then G\"odel’s. (\cite[p.~21]{Smu13})
}
\end{quotation}

Roughly speaking, every Smullyan's system is specified by determining a set of predicates which are finite strings and determining which set of finite strings each predicate names. 
In addition, Smullyan's framework employs two prefixes $\N$ and $\R$ for predicates and special rules regarding the naming relation for the predicates prefixed by these symbols.\footnote{Smullyan adopted the capital letters N and R, while we use $\N$ and $\R$, respectively, in view of the use of lower case letters for symbols.} 
For such a simple system, Smullyan proved the following three theorems.

\begin{thm*}[Theorem F]
Every predicate has a fixed point. 
\end{thm*}

\begin{thm*}[Theorem T]
The set of true sentences is not nameable. 
\end{thm*}

\begin{thm*}[Theorem G]
If a predicate $P$ names a set of true sentences, then there is a sentence X that is undecidable in $P$. 
\end{thm*}

We revisit Smullyan's ``Truth and Provability'' in this paper for three main purposes.
First, we introduce the notion of \textit{Smullyan models} to give a precise definition for Smullyan's framework in order to make it easier to discuss the framework mathematically.
Second, we clarify the relationship between the symbols $\N$ and $\R$ and the three theorems stated above.
For this purpose, in addition to the above three theorems originally considered by Smullyan, we introduce several other properties that would be expected to hold for Smullyan models and analyze their relationships in terms of both implications and non-implications.
Third, we construct Smullyan models based on arithmetical ideas and discuss the correspondence between the properties of these models and those concerning truth and provability in arithmetic.

The organization of the present paper is as follows. 
In Section \ref{SmullyanModels}, we introduce the notion of Smullyan models and reprove above mentioned three Smullyan's theorems according to Smullyan models. 
In Section \ref{properties}, we introduce several properties of Smullyan models and then prove some equivalences between these properties. 
The implications between these properties will be summarized in Figure \ref{Fig1}. 
Section \ref{non-implications} is devoted to proving non-implications between some of these properties by giving several counterexamples of Smullyan models. 
In Section \ref{arithmetic}, we construct two specific Smullyan models $M_{\mathbb{N}}$ and $M_{\PA}$ based on the standard model of arithmetic and Peano arithmetic, respectively. 
Among other things, we show that a stronger version of Theorem T for $M_{\mathbb{N}}$ actually yields the original Tarski's Undefinability Theorem.

\section{Smullyan models and Smullyan's theorems}\label{SmullyanModels}

In this section, we introduce the notion of Smullyan models and reprove Smullyan's three theorems. 

For each non-empty set $\Sigma$ of symbols, let $\Sigma^\ast$ denote the set of all finite strings of the elements of $\Sigma$. 
Let $\epsilon$ denote the empty string and we assume $\epsilon \in \Sigma^\ast$.  
For any $X, Y \in \Sigma^\ast$, let $X Y$ denote the finite string obtained by concatenating $Y$ after the last element of $X$.
For each $X \in \Sigma^\ast$ and $i \in \mathbb{N}$, $X^i$ is inductively defined as follows: $X^0$ is $\epsilon$; and $X^{i+1}$ is $X^i X$. 
So, $X^i$ is $\underbrace{XX\cdots X}_i$.

\begin{dfn}[Smullyan models]
A triple $M = (\Sigma, \Pred, \Phi)$ is said to be a \textit{Smullyan model} if it satisfies the following conditions: 
\begin{itemize}
    \item $\Sigma$ is a non-empty set of symbols. 
    \item $\Pred$ is a subset of $\Sigma^\ast$ satisfying the following requirement: 
    \begin{itemize}
       \item For any $H \in \Pred$ and $X \in \Sigma^\ast \setminus \{\epsilon\}$, we have $HX \notin \Pred$. \hfill ($\dagger$)
    \end{itemize}
    \item $\Phi$ is a function $\Pred \rightarrow \mathcal{P}(\Sigma^\ast)$. 
\end{itemize}
\end{dfn}

For every Smullyan model $M = (\Sigma, \Pred, \Phi)$, we adopt a convention that $\Sigma_M, \Sigma^\ast_M, \Pred_M$ and $\Phi_M$ denote $\Sigma, \Sigma^\ast, \Pred$ and $\Phi$, respectively. 

\begin{dfn}[Predicates and sentences]
Let $M$ be a Smullyan model. 
\begin{itemize}
    \item Every element of $\Pred_M$ is called an \textit{$M$-predicate}. 
    \item A finite string $Y \in \Sigma^\ast_M$ is said to be an \textit{$M$-sentence} if it is of the form $HX$ for some $H \in \Pred_M$ and $X \in \Sigma^\ast_M$. 
    Let $\Stc_M$ denote the set of all $M$-sentences. 
    \item Let $\Stc_M^+ : = \Stc_M \setminus \Pred_M$. 
\end{itemize}
\end{dfn}

Notice that every $M$-predicate $H$ is an $M$-sentence because $H \equiv H \epsilon$. 
Here $X \equiv Y$ means that the finite strings $X$ and $Y$ are identical. 
So, the definition of $\Stc_M^+$ makes sense. 
The following lemma explains why the requirement ($\dagger$) is imposed on the definition of Smullyan models. 

\begin{lem}
Let $M$ be a Smullyan model. 
For each $M$-sentence $S$, the unique $M$-predicate $H$ such that $S$ is of the form $HX$ for some $X \in \Sigma^\ast_M$ is found. 
\end{lem}
\begin{proof}
    Suppose, towards a contradiction, that $M$-sentences $HX$ and $H' X'$ are identical for some distinct $M$-predicates $H$ and $H'$. 
    Without loss of generality, we may assume that $H$ is a proper initial segment of $H'$. 
    Then, we find a non-empty $Y \in \Sigma^\ast_M$ such that $H' \equiv HY$. 
    This violates the requirement ($\dagger$). 
\end{proof}

For any Smullyan model $M$, we say that an $M$-predicate $H$ \textit{names} a subset $V \subseteq \Sigma^\ast_M$ if $V = \Phi_M(H)$. 
For each Smullyan model $M$, an $M$-sentence $HX$ is intended to express the statement that `$X$ is contained in the set of all strings named by the $M$-predicate $H$'. 
Thus, each $M$-sentence is determined to be \textit{true} or \textit{false} depending on whether the intended statement actually holds or not, respectively.

\begin{dfn}
Let $M$ be a Smullyan model. 
\begin{itemize}
    \item $\True_M : = \{HX \in \Stc_M \mid H \in \Pred_M$ and $X \in \Phi_M(H)\}$. 
    \item $\True_M^+ : = \{HX \in \Stc_M^+ \mid H \in \Pred_M$ and $X \in \Phi_M(H)\}$. 
    \item $\False_M : = \Stc_M \setminus \True_M$. 
    \item $\False_M^+ : = \Stc_M^+ \setminus \True_M^+$. 
\end{itemize}
For each $M$-sentence $S$, we write $M \models S$ if $S \in \True_M$. 
\end{dfn}

Notice that $\True_M^+ = \True_M \setminus \Pred_M$ and $\False_M^+ = \False_M \setminus \Pred_M$ hold.

In addition to the basic framework described above, Smullyan considered the system equipped with the two special symbols $\N$ and $\R$.

\begin{dfn}[$\N$-Smullyan models, $\R$-Smullyan models, and $\N \R$-Smullyan models]
Let $M$ be a Smullyan model. 
\begin{itemize}
    \item $M$ is called an \textit{$\N$-Smullyan model} if $\N \in \Sigma_M$ and for each $H \in \Pred_M$, we have $\N H \in \Pred_M$ and
    \[
        \Phi_M(\N H) = \Sigma^\ast_M \setminus \Phi_M(H).
    \]
    \item $M$ is called an \textit{$\R$-Smullyan model} if $\R \in \Sigma_M$ and for each $H \in \Pred_M$, we have $\R H \in \Pred_M$ and
\[
    \Phi_M(\R H) = \{K \in \Pred_M \mid KK \in \Phi_M(H)\}.
\]
    \item $M$ is said to be an \textit{$\N \R$-Smullyan model} if it is both an $\N$-Smullyan model and an $\R$-Smullyan model. 
\end{itemize}
\end{dfn}

For each $\N$-Smullyan model $M$, the symbol $\N$ behaves as the negation, that is, it is easily shown that for any $M$-sentence $S$, we have that $M \models \N S$ if and only if $M \not \models S$. 
The symbol $\R$ was used by Smullyan ``to suggest the word \textit{repeat}''.

\begin{dfn}
Let $M$ be a Smullyan model. 
We say that $S \in \Stc_M^+$ is an \textit{$M$-fixed point} of an $M$-predicate $H$ if the following equivalence holds: 
\[
    M \models S \iff M \models HS.
\]
\end{dfn}

\begin{thm}[Fixed Point Theorem {\cite[Theorem F]{Smu13}}]\label{THM_FPT}
Let $M$ be an $\R$-Smullyan model. 
For any $M$-predicate $H$, there exists an $M$-fixed point of $H$.
\end{thm}
\begin{proof}
For each $H \in \Pred_M$, the following equivalences show that $\R H \R H \in \Stc_M^+$ is an $M$-fixed point of $H$: 
\[
M \models \R H \R H \iff \R H \in \Phi_M(\R H) \iff \R H \R H \in \Phi_M(H) \iff M \models H \R H \R H. \qedhere
\]
\end{proof}

\begin{thm}[Tarski's Undefinability Theorem {\cite[Theorem T]{Smu13}}]\label{THM_TTARSKI}
For any $\N \R$-Smullyan model $M$, there is no $M$-predicate that names $\True_M$.
\end{thm}
\begin{proof}
Let $H$ be any $M$-predicate. 
By the Fixed Point Theorem, we find an $M$-fixed point $S$ of the $M$-predicate $\N H$. 
Then, we have
\begin{align*}
    S \in \True_M & \iff M \models S \iff M \models \N HS\\
    &\iff M \not \models H S \iff S \notin \Phi_M(H).
\end{align*}
These equivalences show that $H$ does not name $\True_M$.
\end{proof}

\begin{thm}[G\"odel's First Incompleteness Theorem {\cite[Theorem G]{Smu13}}]\label{THM_NGODEL}
Let $M$ be an $\N \R$-Smullyan model. 
For any $M$-predicate $H$ satisfying $\Phi_M(H) \subseteq \True_M$, there exists an $M$-sentence $S$ such that $S \notin \Phi_M(H)$ and $\N S \notin \Phi_M(H)$.
\end{thm}
\begin{proof}
Suppose that an $M$-predicate $H$ satisfies $\Phi_M(H) \subseteq \True_M$. 
By the Fixed Point Theorem, we find an $M$-fixed point $S$ of the $M$-predicate $\N H$.
We have 
\begin{align*}
    S \in \True_M & \iff M \models S \iff M \models \N HS \\
    & \iff M \not \models HS \iff S \notin \Phi_M(H).
\end{align*}
By combining these equivalences with the supposition $\Phi_M(H) \subseteq \True_M$, we get $S \in \True_M$ and $S \notin \Phi_M(H)$.
Then, we have $\N S \notin \True_M$, which implies $\N S \notin \Phi_M(H)$. 
We have shown that the $M$-sentence $S$ witnesses the theorem. 
\end{proof}

\section{Properties of Smullyan models}\label{properties}

In this section, we introduce several properties of Smullyan models. 
We then prove some equivalences between these properties. 
We would like to mention here that the equivalence of the Fixed Point Theorem, Tarski's theorem, and G\"odel's theorem in arithmetic was discussed by Salehi \cite{sal22,Sal24}.

\begin{dfn}[Properties of Smullyan models]\label{def:properties}
We consider the following properties of Smullyan models $M$:
\begin{itemize}
    \item ($\FPT$) Every $M$-predicate has an $M$-fixed point.
    \item ($\TTarski$) There is no $M$-predicate that names $\True_M$.
    \item ($\FTarski$) There is no $M$-predicate that names $\False_M$.
    \item ($\STTarski$) There is no $M$-predicate $H$ such that $\True_M^+ = \Phi_M(H) \cap \Stc_M^+$.
    \item ($\SFTarski$) There is no $M$-predicate $H$ such that $\False_M^+ = \Phi_M(H) \cap \Stc_M^+$.
    \item ($\PGodel$) For any $M$-predicate $H$ satisfying $\Phi_M(H) \subseteq \True_M$, there exists $S \in \Stc_M$ such that $M \models S$ and $S \notin \Phi_M(H)$.
    \item ($\SPGodel$) For any $M$-predicate $H$ satisfying $\Phi_M(H) \cap \Stc_M^+ \subseteq \True_M^+$, there exists $S \in \Stc_M^+$ such that $M \models S$ and $S \notin \Phi_M(H)$.
\end{itemize}

We also consider the following properties of $\N$-Smullyan models $M$:

\begin{itemize}
    \item ($\NGodel$) For any $M$-predicate $H$ satisfying $\Phi_M(H) \subseteq \True_M$, there exists $S \in \Stc_M$ such that $S \notin \Phi_M(H)$ and $\N S \notin \Phi_M(H)$.
    \item ($\SNGodel$) For any $M$-predicate $H$ satisfying $\Phi_M(H) \cap \Stc_M^+ \subseteq \True_M^+$, there exists $S \in \Stc_M^+$ such that $S \notin \Phi_M(H)$ and $\N S \notin \Phi_M(H)$.
\end{itemize}
\end{dfn}

$\FPT$ stands for `Fixed Point Theorem' and Smullyan's Theorem F (Theorem \ref{THM_FPT}) states that every $\R$-Smullyan model satisfies $\FPT$. 
$\TTarski$ and $\FTarski$ state that Tarski's Undefinability Theorem holds for $M$ with respect to $\True_M$ and $\False_M$, respectively. 
Theorem \ref{THM_TTARSKI} states that every $\N \R$-Smullyan model satisfies $\TTarski$. 
Although $\TTarski$ and $\FTarski$ seem to be equivalent, indeed they are not.
In fact, we will prove in the next section that these properties are incomparable even if we consider $\N$-Smullyan models (Propositions \ref{FT-non-TT} and \ref{TT-non-FT}).
This incomparability is caused by the reason that, in general, the set named by an $M$-predicate $H$ may contain finite strings that are not $M$-sentences, and such $H$ trivially names neither $\True_M$ nor $\False_M$. 
Thus, we consider $\STTarski$ and $\SFTarski$, which are stronger versions of $\TTarski$ and $\FTarski$, respectively, in which this triviality is removed.
The reason why we defined these strong properties using $\Stc^+_M$ rather than $\Stc_M$ is to yield meaningful properties of arithmetic in Section \ref{arithmetic}.
In the next section, we will prove that the stronger versions are actually strictly stronger than the original ones. 
We also prove in this section that $\STTarski$ and $\SFTarski$ are equivalent for any $\N$-Smullyan models (Proposition \ref{SFT-STT}). 
Interestingly, this stronger version $\SFTarski$ of $\FTarski$ is equivalent to $\FPT$ for any Smullyan model (Proposition \ref{FPT-SFT}).

Theorem \ref{THM_NGODEL} states that every $\N \R$-Smullyan model satisfies $\NGodel$, where G1 stands for `G\"odel's 1st Theorem'.
However, $\NGodel$ needs to consider $\N$-Smullyan models to state it, so it is a bit awkward for our purposes of analyzing the general situation of Smullyan models.
For this reason, we introduce a new property $\PGodel$ which corresponds to a version of G\"odel's First Incompleteness Theorem stating that `every computable sound extension of Peano arithmetic has a true but unprovable sentence'. 
Here `m' stands for `modified'. 
We prove in this section that $\NGodel$ and $\PGodel$ are equivalent for any $\N$-Smullyan model (Proposition \ref{PG-NG}). 
Furthermore, we prove that $\TTarski$ and $\PGodel$ are equivalent for every Smullyan model (Proposition \ref{TT-PG}). 
The properties $\SNGodel$ and $\SPGodel$ are stronger versions of $\NGodel$ and $\PGodel$, respectively, in which the triviality is removed.

\subsection{Equivalences between the properties}

We show several equivalences between the properties introduced above.
The results of this section are summarized in Figure \ref{Fig1}. 
In conclusion of this section, it is sufficient to consider the four properties $\TTarski$, $\FTarski$, $\STTarski$, and $\SFTarski$ when dealing with the properties we have introduced.

\begin{figure}[ht]
\centering
\begin{tikzpicture}
\node (FPT) at (0,2) {$\FPT$};
\node (SFT) at (2.5, 2) {$\SFTarski$};
\node (FT) at (2.5, 0) {$\FTarski$};
\node (STT) at (5, 2) {$\STTarski$};
\node (TT) at (5, 0) {$\TTarski$};
\node (SPG) at (7.5, 2) {$\SPGodel$};
\node (PG) at (7.5, 0) {$\PGodel$};
\node (SNG) at (10, 2) {$\SNGodel$};
\node (NG) at (10, 0) {$\NGodel$};

\draw [<->, double] (FPT)--(SFT);
\draw [->, double] (SFT)--(FT);
\draw [<->, double] (SFT)--(STT);
\draw [->, double] (STT)--(TT);
\draw [<->, double] (STT)--(SPG);
\draw [<->, double] (TT)--(PG);
\draw [->, double] (SPG)--(PG);
\draw [<->, double] (SPG)--(SNG);
\draw [<->, double] (PG)--(NG);
\draw [->, double] (SNG)--(NG);

\draw (1,2.2) node[above]{Prop.~\ref{FPT-SFT}};
\draw (3.7,2.2) node[above]{Prop.~\ref{SFT-STT}};
\draw (6.2,2.2) node[above]{Prop.~\ref{TT-PG}};
\draw (6.2,0.2) node[above]{Prop.~\ref{TT-PG}};
\draw (8.7,2.2) node[above]{Prop.~\ref{PG-NG}};
\draw (8.7,0.2) node[above]{Prop.~\ref{PG-NG}};

\draw (3.75,2) node[below]{$\N$};
\draw (8.75,2) node[below]{$\N$};
\draw (8.75,0) node[below]{$\N$};

\end{tikzpicture}
\caption{Implications between the properties}\label{Fig1}
\end{figure}
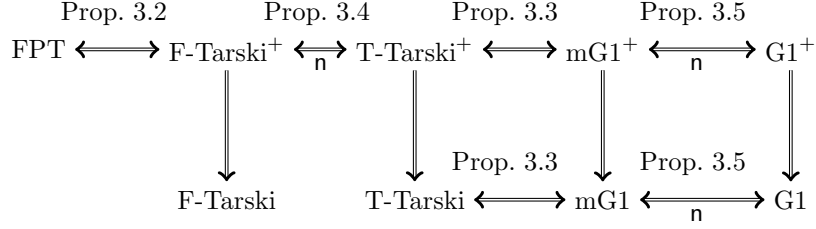

Notice that what Smullyan achieved can also be understood through the equivalences shown in Figure \ref{Fig1}. 
For, Theorem \ref{THM_FPT} states that every $\R$-Smullyan model satisfies $\FPT$, and hence Figure \ref{Fig1} shows that every $\N \R$-Smullyan model satisfies all the properties indicated in the figure. 

\begin{prop}\label{FPT-SFT}
$\FPT$ and $\SFTarski$ are equivalent for any Smullyan model.
\end{prop}
\begin{proof}
Let $M$ be any Smullyan model. 
This proposition is proved by the following equivalences: 
\begin{align*}
\FPT 
\iff  & (^\forall H \in\Pred_M)\left(^\exists S \in \Stc_M^+ \right) \left( M \models S \iff M \models HS \right) 
\\ \iff & (^\forall H \in\Pred_M) \left(^\exists S \in \Stc_M^+ \right) \left(S \in \True_M \iff  S \in \Phi_M(H) \right) 
\\ \iff & (^\forall H \in\Pred_M) \left(^\exists S \in \Stc_M^+ \right) \\
    & \qquad \qquad \qquad \left(S \in \False_M^+ \iff S \notin \Phi_M(H) \cap \Stc_M^+ \right) 
\\ \iff & (^\forall H \in\Pred_M) \left(\False_M^+ \ne \Phi_M(H) \cap \Stc_M^+ \right) 
\\ \iff & \SFTarski. \qedhere
\end{align*}
\end{proof}

\begin{prop}\label{TT-PG}\leavevmode
\begin{enumerate}
    \item $\TTarski$ and $\PGodel$ are equivalent for any Smullyan model.
    \item $\STTarski$ and $\SPGodel$ are equivalent for any Smullyan model.
\end{enumerate}
\end{prop}
\begin{proof}
We only prove the first clause. 
The second clause is proved in the similar way. 
Let $M$ be any Smullyan model. 
It is easy to see that for each $H \in \Pred_M$, the condition $\True_M \neq \Phi_M(H)$ is equivalent to 
\begin{center}
    ``$\Phi_M(H) \subseteq \True_M \Rightarrow \ ^\exists S \in \Stc_M (M \models S\ \&\ S \notin \Phi_M(H)$)''. 
\end{center}
This shows that $\TTarski$ and $\PGodel$ are equivalent for $M$. 
\end{proof}

\begin{prop}\label{SFT-STT}
$\SFTarski$ and $\STTarski$ are equivalent for any $\N$-Smullyan model. 
\end{prop}
\begin{proof}
Let $M$ be any $\N$-Smullyan model. 
Notice that $\Phi_M(H) = \Phi_M(\N\N H)$ holds for any $H \in \Pred_M$. 
Then, this proposition is shown by the following equivalences: 
\begin{align*}
\SFTarski
\iff & (^\forall H \in\Pred_M)\left(\False_M^+ \neq \Phi_M(H)\cap  \Stc_M^+ \right)\\
\iff & (^\forall H \in\Pred_M) \bigl[\left(\False_M^+ \neq \Phi_M(H)\cap  \Stc_M^+ \right) \\
    & \qquad \qquad \qquad \qquad \&\  \left(\False_M^+ \neq \Phi_M(\N H)\cap  \Stc_M^+ \right) \bigr]\\
\iff & (^\forall H \in\Pred_M) \bigl[\left(\False_M^+ \neq (\Sigma^\ast_M \setminus \Phi(\N H))\cap  \Stc_M^+\right) \\
& \qquad\qquad \qquad \qquad \ \&\  \left(\False_M^+ \neq (\Sigma^\ast_M \setminus \Phi_M(H))\cap  \Stc_M^+ \right) \bigr]\\
\iff & (^\forall H \in\Pred_M) \bigl[\left( \Stc_M^+ \setminus \True_M^+ \neq  \Stc_M^+ \setminus \Phi_M(\N H) \right) \\
& \qquad\qquad \qquad \qquad \ \&\ \left( \Stc_M^+ \setminus \True_M^+ \neq  \Stc_M^+ \setminus \Phi_M(H) \right) \bigr]\\
\iff & (^\forall H \in\Pred_M) \bigl[\left(\True_M^+ \neq \Phi_M(H)\cap  \Stc_M^+ \right) \\
    & \qquad \qquad \qquad \qquad \&\  \left(\True_M^+ \neq \Phi_M(\N H)\cap  \Stc_M^+ \right) \bigr]\\
\iff & (^\forall H \in\Pred_M) \left(\True_M^+  \neq \Phi_M(H)  \cap  \Stc_M^+\right)\\
\iff &  \STTarski. \qedhere
\end{align*}
\end{proof}

\begin{prop}\label{PG-NG}\leavevmode
\begin{enumerate}
    \item $\PGodel$ and $\NGodel$ are equivalent for any $\N$-Smullyan model.
    \item $\SPGodel$ and $\SNGodel$ are equivalent for any $\N$-Smullyan model.
\end{enumerate}
\end{prop}
\begin{proof}
We only prove the first clause. 
The second clause is proved in a similar way. 
Let $M$ be any $\N$-Smullyan model. 

$(\PGodel \Rightarrow \NGodel)$: Suppose that $M$ satisfies $\PGodel$. 
Let $H$ be any $M$-predicate such that $\Phi_M(H) \subseteq \True_M$. 
By $\PGodel$, there is some $S \in \Stc_M$ such that $M \models S$ and $S \notin \Phi_M(H)$. 
Then, we have $M \not \models \N S$ and hence $\N S \notin \True_M$. 
Since $\Phi_M(H) \subseteq \True_M$, we obtain $\N S \notin \Phi_M(H)$. 
We have shown that $S$ witnesses $\NGodel$ for $M$. 

$(\NGodel \Rightarrow \PGodel)$: Suppose that $M$ satisfies $\NGodel$. 
Let $H$ be any $M$-predicate such that $\Phi_M(H) \subseteq \True_M$. 
By $\NGodel$, there is some $S \in \Stc_M$ such that $S \notin \Phi_M(H)$ and $\N S \notin \Phi_M(H)$. 
Depending on whether $M \models S$ or $M \models \N S$, we have that $S$ or $\N S$ is a witness of $\PGodel$ for $M$, respectively.
\end{proof}

\section{Non-implications}\label{non-implications}

As mentioned before, the results we presented in the previous section indicate that it suffices to analyze the relationship between the four properties $\SFTarski$, $\FTarski$, $\STTarski$, and $\TTarski$. 
The situation of the implications between the combinations of these properties is visualized in Figure \ref{Fig2}. 

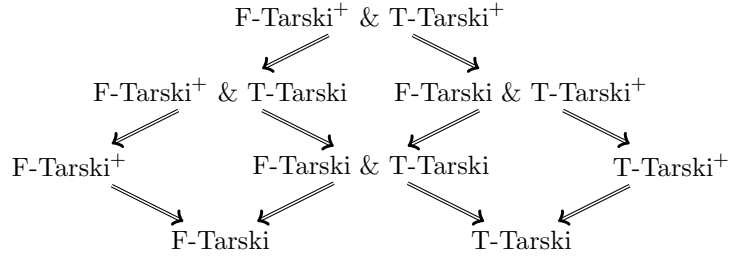
\begin{figure}[ht]
\centering
\begin{tikzpicture}
\node (SFTSTT) at (4, 3) {$\SFTarski\ \&\ \STTarski$};
\node (SFTTT) at (2, 2) {$\SFTarski\ \&\ \TTarski$};
\node (FTSTT) at (6, 2) {$\FTarski\ \&\ \STTarski$};
\node (SFT) at (0, 1) {$\SFTarski$};
\node (FTTT) at (4, 1) {$\FTarski\ \&\ \TTarski$};
\node (STT) at (8, 1) {$\STTarski$};
\node (FT) at (2, 0) {$\FTarski$};
\node (TT) at (6, 0) {$\TTarski$};

\draw [->, double] (SFTSTT)--(SFTTT);
\draw [->, double] (SFTSTT)--(FTSTT);
\draw [->, double] (SFTTT)--(SFT);
\draw [->, double] (SFTTT)--(FTTT);
\draw [->, double] (FTSTT)--(FTTT);
\draw [->, double] (FTSTT)--(STT);
\draw [->, double] (SFT)--(FT);
\draw [->, double] (FTTT)--(FT);
\draw [->, double] (FTTT)--(TT);
\draw [->, double] (STT)--(TT);
\end{tikzpicture}
\caption{Implications between the properties}\label{Fig2}
\end{figure}

In this section, we prove that in general no more arrows can be added into Figure \ref{Fig2}. 
Concretely we prove the following non-implications. 

\begin{itemize}
    \item $\SFTarski \overset{\R}{\not \Rightarrow} \TTarski$ (Proposition \ref{SFT-non-TT}). 
    \item $\STTarski \not \Rightarrow \FTarski$ (Proposition \ref{STT-non-FT}).
    \item $\FTarski \ \&\ \TTarski \overset{\N}{\not \Rightarrow} \SFTarski$ (Proposition \ref{FTTT-non-SFTSTT}).
    \item $\FTarski \ \&\ \TTarski \overset{\N}{\not \Rightarrow} \STTarski$ (Proposition \ref{FTTT-non-SFTSTT}).
    \item $\SFTarski \ \&\ \TTarski \overset{\R}{\not \Rightarrow} \STTarski$ (Proposition \ref{SFTTT-non-STT}).
    \item $\FTarski \ \&\ \STTarski \not \Rightarrow \SFTarski$ (Proposition \ref{STTFT-non-SFT}).
    \item $\FTarski \overset{\N}{\not \Rightarrow} \TTarski$ (Proposition \ref{FT-non-TT}). 
    \item $\TTarski \overset{\N}{\not \Rightarrow} \FTarski$ (Proposition \ref{TT-non-FT}). 
\end{itemize}
The relation $\overset{\R}{\not \Rightarrow}$ (resp.~$\overset{\N}{\not \Rightarrow}$) indicates that the non-implication is shown by giving a counter-model which is an $\R$-Smullyan model (resp.~$\N$-Smullyan model). 
As a consequence of these non-implications, no more arrows can be added into Figure \ref{Fig1} as well. 

Our counter-models presented in this section are restricted and tractable versions of the general Smullyan models, and we first introduce the notion of these \textit{simple models}. 
We fix a symbol $\sharp$, which is different from both $\N$ and $\R$. 

\begin{dfn}[Simple models]
A tuple $M = (\Sigma, \Phi)$ is said to be a \textit{simple model} if it satisfies the following conditions: 
\begin{itemize}
    \item $\Sigma$ is a non-empty set of symbols such that $\sharp \in \Sigma$. 
    \item Let $\Pred_M : = \{X \sharp \mid X \in \Sigma^\ast$ and $X$ does not contain $\sharp\}$. 
    \item $\Phi$ is a function $\Pred_M \rightarrow \mathcal{P}(\Sigma^\ast)$. 
\end{itemize}
\end{dfn}

It is easy to show that every simple model satisfies the requirement ($\dagger$) in the definition of Smullyan models. 
So, we have the following proposition. 

\begin{prop}
For every simple model $M = (\Sigma, \Phi)$, the triple $(\Sigma, \Pred_M, \Phi)$ is a Smullyan model. 
\end{prop}

Thus, in the following, we will deal with simple models as Smullyan models.
The following proposition says that for every simple model, it is very easy to determine whether a given finite string is a sentence or not. 
In this sense, simple models are easy to handle.

\begin{prop}\label{simple_stc}
For every simple model $M = (\Sigma, \Phi)$, we have
\[
    \Stc_M =  \{X \in \Sigma^\ast \mid X \ \text{contains at least one}\ \sharp\}.
\]
\end{prop}

We are ready to give our counter-models of several implications. 

\begin{prop}\label{SFT-non-TT}
There is an $\R$-simple model which satisfies $\SFTarski$ but does not satisfy $\TTarski$.
\end{prop}
\begin{proof}
Let $M$ be the $\R$-simple model defined as follows:
\begin{itemize}
    \item $\Sigma_M : = \{\R, \sharp\}$,
    \item $\Phi_M(\sharp) : = \emptyset$. 
\end{itemize}
It is easy to show that $\Pred_M = \{\R^i \sharp \mid i \in \mathbb{N} \}$. 
Then, by the definition of $\R$-Smullyan models, it can be shown that $\Phi_M(H) = \emptyset$ for all $H \in \Pred_M$. 
Therefore, we obtain $\True_M = \emptyset = \Phi_M(\sharp)$. 
This means that $M$ does not satisfy $\TTarski$. 
By Theorem \ref{THM_FPT} and Proposition \ref{FPT-SFT}, $M$ satisfies $\SFTarski$. 
\end{proof}

We also give a proof of Proposition \ref{SFT-non-TT} with a counter-model having a slightly non-trivial function $\Phi$.

\begin{proof}[Alternative proof of Proposition \ref{SFT-non-TT}]
Let $M$ be the $\R$-simple model defined as follows:
\begin{itemize}
    \item $\Sigma_M : = \{\R, \sharp\}$,
    \item $\Phi_M(\sharp) : = \{\sharp^i \R \,\sharp \, \R \, \sharp \mid i \in \mathbb{N}\}$. 
\end{itemize}
As above, we have that $M$ satisfies $\SFTarski$ and $\Pred_M = \{\R^i \sharp \mid i \in \mathbb{N} \}$. 
Since the only element of $\Phi_M(\sharp)$ which is of the form $KK$ for some $M$-predicate $K$ is $\R \, \sharp \, \R \, \sharp$, we have $\Phi_M(\R \, \sharp) = \{\R \, \sharp\}$. 
Then, it is easy to see that $\Phi_M(\R^i \sharp) = \emptyset$ for all $i \geq 2$. 
Therefore, we obtain
\[
    \True_M = \{\sharp \, \sharp^{i} \R\, \sharp\, \R\, \sharp \mid i \in \mathbb{N}\} \cup \{\R\, \sharp\, \R\, \sharp\} = \{\sharp^{i} \R\, \sharp\, \R\, \sharp \mid i \in \mathbb{N}\} = \Phi_M(\sharp).
\]
This implies that $M$ does not satisfy $\TTarski$. 
\end{proof}

\begin{prop}\label{STT-non-FT}
There is a simple model which satisfies $\STTarski$ but does not satisfy $\FTarski$.
\end{prop}
\begin{proof}
Let $M$ be the simple model defined as follows:
\begin{itemize}
    \item $\Sigma_M : = \{\sharp\}$,
    \item $\Phi_M(\sharp) : = \{\sharp^{2i+1} \mid i \in \mathbb{N}\}$. 
\end{itemize}
We have $\Pred_M = \{\sharp\}$, $\Stc_M = \{\sharp^i \mid i \geq 1\}$, $\Stc_M^+ = \{\sharp^i \mid i \geq 2\}$, $\True_M = \True_M^+ = \{\sharp^{2i+2} \mid i \in \mathbb{N}\}$, and $\False_M = \{\sharp^{2i+1} \mid i \in \mathbb{N}\}$. 
Since $\True_M^+ \neq \Phi_M(\sharp) \cap \Stc_M^+$, we have that $M$ satisfies $\STTarski$. 
On the other hand, since $\False_M = \Phi_M(\sharp)$, we get that $M$ does not satisfy $\FTarski$. 
\end{proof}

\begin{prop}\label{FTTT-non-SFTSTT}
There is an $\N$-simple model which satisfies both $\FTarski$ and $\TTarski$ but neither $\SFTarski$ nor $\STTarski$.
\end{prop}
\begin{proof}
Let $M$ be the $\N$-simple model defined as follows:
\begin{itemize}
    \item $\Sigma_M : = \{\N, \sharp\}$,
    \item For each $X \in \Sigma_M^\ast$, let $n(X)$ be the number of occurrences of $\N$ in $X$, 
    \item $\Phi_M(\sharp) : = \{X \in \Sigma_M^\ast \mid n(X)$ is even$\}$. 
\end{itemize}
We have $\Pred_M = \{\N^i \sharp \mid i \in \mathbb{N}\}$. 
For each $i \in \mathbb{N}$, it is shown that
\begin{enumerate}
    \item $\Phi_M(\N^{2i} \sharp) = \{X \in \Sigma_M^\ast \mid n(X)$ is even$\}$, 
    \item $\Phi_M(\N^{2i+1} \sharp) = \{X \in \Sigma_M^\ast \mid n(X)$ is odd$\}$. 
\end{enumerate}
Since the strings $\epsilon$ and $\N$ are not $M$-sentences, $\epsilon \in \Phi_M(\N^{2i} \sharp)$, and $\N \in \Phi_M(\N^{2i+1} \sharp)$, we obtain that $\Phi_M(H)$ coincides with neither $\False_M$ nor $\True_M$ for every $H \in \Pred_M$. 
This means that $M$ satisfies both $\FTarski$ and $\TTarski$. 

On the other hand, we have
\begin{align*}
    \True_M^+ & = (\{\N^{2i} \sharp \, X \mid i \in \mathbb{N}\ \&\ n(X)\ \text{is even}\} \\
    & \qquad \cup \{\N^{2i+1} \sharp \, X \mid i \in \mathbb{N}\ \&\ n(X)\ \text{is odd}\} ) \setminus \Pred_M \\
    & = \{X \in \Stc_M \mid n(X)\ \text{is even}\} \setminus \Pred_M \\
    & = \{X \in \Stc_M^+ \mid n(X)\ \text{is even}\} \\
    & = \Phi_M(\sharp) \cap \Stc_M^+. 
\end{align*}
Therefore, $M$ does not satisfy $\STTarski$. 
Since $M$ is an $\N$-Smullyan model, by Proposition \ref{SFT-STT}, $M$ also does not satisfy $\SFTarski$. 
\end{proof}

\begin{prop}\label{SFTTT-non-STT}
There is an $\R$-simple model which satisfies $\SFTarski$ and $\TTarski$ but does not satisfy $\STTarski$.
\end{prop}
\begin{proof}
Let $M$ be the $\R$-simple model defined as follows:
\begin{itemize}
    \item $\Sigma_M : = \{\R, \sharp\}$,
    \item $\Phi_M(\sharp) : = \{\sharp^{i}, \sharp^i \R \, \sharp \, \sharp \mid i \in \mathbb{N}\}$. 
\end{itemize}
We have that $M$ satisfies $\SFTarski$ and $\Pred_M = \{\R^j \sharp \mid j \in \mathbb{N}\}$. 
Since the only element of $\Phi_M(\sharp)$ which is of the form $KK$ for some $K \in \Pred_M$ is $\sharp\, \sharp$, we have $\Phi_M(\R \, \sharp) = \{\sharp\}$. 
Then, it is easy to see that $\Phi_M(\R^j \sharp) = \emptyset$ for all $j \geq 2$. 
Therefore, we get
\[
    \True_M = \{\sharp^{i+1}, \sharp^{i+1} \R \, \sharp \, \sharp \mid i \in \mathbb{N}\} \cup \{\R\, \sharp \, \sharp\} = \{\sharp^{i+1}, \sharp^{i} \R \, \sharp \, \sharp \mid i \in \mathbb{N}\}.
\]
Since $\True_M \neq \Phi_M(H)$ for all $H \in \Pred_M$, we have that $M$ satisfies $\TTarski$. 
On the other hand, since $\epsilon, \sharp \notin\Stc_M^+$, we get $\True_M^+ = \Phi_M(\sharp) \cap \Stc_M^+$. 
Hence, $M$ does not satisfy $\STTarski$. 
\end{proof}

\begin{prop}\label{STTFT-non-SFT}
There is a simple model which satisfies $\STTarski$ and $\FTarski$ but does not satisfy $\SFTarski$.
\end{prop}
\begin{proof}
Let $M$ be the simple model defined as follows:
\begin{itemize}
    \item $\Sigma_M : = \{\sharp\}$,
    \item $\Phi_M(\sharp) : = \{\sharp^{2i} \mid i \in \mathbb{N}\}$. 
\end{itemize}
We have that $\Pred_M = \{\sharp\}$, $\Stc_M = \{\sharp^i \mid i \geq 1\}$, $\Stc_M^+ = \{\sharp^i \mid i\geq 2\}$, $\True_M^+ = \{\sharp^{2i+1} \mid i \geq 1\}$, and $\False_M = \False_M^+ = \{\sharp^{2i+2}\mid i \in \mathbb{N}\}$. Since $\True_M^+ \neq \Phi_M(\sharp) \cap \Stc_M^+$, we have that $M$ satisfies $\STTarski$. 
Since $\epsilon \in \Phi_M(\sharp)$ and $\epsilon \notin \Stc_M$, we have $\False_M \neq \Phi_M(\sharp)$, that is, $M$ satisfies $\FTarski$. 
On the other hand, we get $\False_M^+ = \Phi_M(\sharp) \cap \Stc_M^+$. 
Hence, $M$ does not satisfy $\SFTarski$.
\end{proof}

\begin{prop}\label{FT-non-TT}
There is an $\N$-simple model which satisfies $\FTarski$ but does not satisfy $\TTarski$.
\end{prop}
\begin{proof}
Let $M$ be the $\N$-simple model defined as follows:
\begin{itemize}
    \item $\Sigma_M : = \{\N, \sharp\}$,
    \item $\Phi_M(\sharp) : = \{X \, \sharp\, \N^i \mid i \in \mathbb{N}$ and $n(X)$ is odd$\}$, where $n(X)$ is defined as in the proof of Proposition \ref{FTTT-non-SFTSTT}. 
\end{itemize}
We have $\Pred_M = \{\N^i \sharp \mid i \in \mathbb{N}\}$, $\Stc_M = \{X \, \sharp \, \N^i \mid i \in \mathbb{N}$ and $X \in \Sigma^\ast_M\}$, and $\Sigma^\ast_M \setminus \Stc_M = \{\N^i \mid i \in \mathbb{N}\}$. 
It is shown that for each $j \in \mathbb{N}$,
\begin{enumerate}
    \item $\Phi_M(\N^{2j} \sharp) = \{X \, \sharp\, \N^i \mid i \in \mathbb{N}$ and $n(X)$ is odd$\}$, 
    \item $\Phi_M(\N^{2j+1} \sharp) = \{\N^i, X \, \sharp\, \N^i \mid i \in \mathbb{N}$ and $n(X)$ is even$\}$. 
\end{enumerate}

We have 
\begin{align*}
    \True_M & = \{\N^{2j} \sharp \, X \, \sharp\, \N^i \mid i,j \in \mathbb{N}\ \&\ n(X)\ \text{is odd}\}\\
        & \qquad  \cup \{\N^{2j+1} \sharp \, \N^i, \N^{2j+1} \sharp \, X  \,  \sharp \, \N^i \mid i, j \in \mathbb{N}\ \&\ n(X)\ \text{is even}\}\\
    & = \{X \, \sharp \, \N^i \mid i \in \mathbb{N}\ \text{and}\ n(X)\ \text{is odd}\}\\
    & = \Phi_M(\sharp). 
\end{align*}
Therefore, $M$ does not satisfy $\TTarski$. 

Also, since
\[
    \False_M = \{X \, \sharp \, \N^i \mid i \in \mathbb{N}\ \text{and}\ n(X)\ \text{is even}\}, 
\]
it is easy to see that for every $H \in \Pred_M$, we have $\False_M \neq \Phi_M(H)$. 
This means that $M$ satisfies $\FTarski$. 
\end{proof}

\begin{prop}\label{TT-non-FT}
There is an $\N$-simple model which satisfies $\TTarski$ but does not satisfy $\FTarski$.
\end{prop}
\begin{proof}
Let $M$ be the $\N$-simple model defined as follows:
\begin{itemize}
    \item $\Sigma_M : = \{\N, \sharp\}$,
    \item $\Phi_M(\sharp) : = \{\N^i, X \, \sharp\, \N^i \mid i \in \mathbb{N}$ and $n(X)$ is odd$\}$. 
\end{itemize}
We have $\Pred_M = \{\N^i \sharp \mid i \in \mathbb{N}\}$, $\Stc_M = \{X \, \sharp \, \N^i \mid i \in \mathbb{N}$ and $X \in \Sigma^\ast_M\}$, and $\Sigma^\ast_M \setminus \Stc_M = \{\N^i \mid i \in \mathbb{N}\}$. 
It is shown that for each $j \in \mathbb{N}$,
\begin{enumerate}
    \item $\Phi_M(\N^{2j} \sharp) = \{\N^i, X \, \sharp\, \N^i \mid i \in \mathbb{N}$ and $n(X)$ is odd$\}$, 
    \item $\Phi_M(\N^{2j+1} \sharp) = \{X \, \sharp\, \N^i \mid i \in \mathbb{N}$ and $n(X)$ is even$\}$. 
\end{enumerate}

We have 
\begin{align*}
    \True_M & = \{\N^{2j} \sharp \, \N^i, \N^{2j} \sharp \, X \, \sharp\, \N^i \mid i,j \in \mathbb{N}\ \&\ n(X)\ \text{is odd}\}\\
        & \qquad  \cup \{\N^{2j+1} \sharp \, X  \,  \sharp \, \N^i \mid i, j \in \mathbb{N}\ \&\ n(X)\ \text{is even}\}\\
    & = \{X \, \sharp \, \N^i \mid i \in \mathbb{N}\ \text{and}\ n(X)\ \text{is odd}\}.
\end{align*}
Since it is easy to see that $\True_M \neq \Phi_M(H)$ for every $H \in \Pred_M$, we have that $M$ satisfy $\TTarski$. 

On the other hand, since
\[
    \False_M = \{X \, \sharp \, \N^i \mid i \in \mathbb{N}\ \text{and}\ n(X)\ \text{is even}\} = \Phi_M(\N \, \sharp), 
\]
we obtain that $M$ does not satisfies $\FTarski$. 
\end{proof}

\section{Smullyan models based on arithmetic}\label{arithmetic}

In this section, we discuss the property of first-order arithmetic through the analysis of Smullyan models. 
We construct two specific Smullyan models $M_{\mathbb{N}}$ and $M_{\PA}$ based on arithmetic and we then discuss the correspondence between the properties of these models and those concerning truth and provability in arithmetic.

Let $\mathcal{L}_A$ be the language of first-order arithmetic. 
We may assume that $\mathcal{L}_A$ contains the corresponding function symbol for each primitive recursive function. 
Let $\mathbb{N}$ denote the $\mathcal{L}_A$-structure that is the standard model of first-order arithmetic. 
Peano arithmetic $\PA$ is the $\mathcal{L}_A$-theory consisting of basic axioms of arithmetic and the axiom scheme of induction. 
For each $n \in \mathbb{N}$, let $\overline{n}$ denote the canonical closed $\mathcal{L}_A$-term whose value is $n$. 
We fix a natural G\"odel numbering of symbols and formulas in $\mathcal{L}_A$ and let $\gn{\varphi}$ denote the G\"odel number of an $\mathcal{L}_A$-formula $\varphi$. 
For a variable $v$, we say that an $\mathcal{L}_A$-formula $\varphi$ is a \textit{$v$-formula} if $\varphi$ contains no free variables other than $v$. 
A $v$-formula having the G\"odel number $i$ is denoted by $\varphi_i(v)$. 
Also, let $\Gamma : = \{i \in \mathbb{N} \mid i$ is the G\"odel number of some $v$-formula$\}$. 
Let $d$ be the primitive recursive function satisfying the following equality: 
\begin{align*}
d(i) = 
\begin{cases}
\gn{\varphi_i(\num{i})} & \text{if } i \in \Gamma,\\
0 & \text{if}\ i \notin \Gamma.
\end{cases}
\end{align*}
It is allowed that a $v$-formula $\varphi_i$ may contain no free variables. 
In this case, we have $d(i) = \gn{\varphi_i}$. 
As noted above, we may assume that $\PA$ has the function symbol $d$ corresponding to this primitive recursive function. 
For the sake of simplicity, for each $\mathcal{L}_A$-formula $\psi$, we will identify $\num{\gn{\psi}}$ and $\gn{\psi}$ in this section.

\subsection{Arithmetic-based Smullyan frames}

We prepare an infinite sequence $\{a_i \mid i \in \mathbb{N}\}$ of fresh symbols. 
We are ready to bring arithmetic into Smullyan models. 
In the next two subsections, we will define two Smullyan models $M_{\mathbb{N}}$ and $M_{\PA}$ based on the standard model $\mathbb{N}$ of arithmetic and Peano arithmetic $\PA$, respectively. 
For this purpose, in this subsection, we introduce two frames $F_{\R}$ and $F_{\N \R}$ based on arithmetic as a small step.



\begin{dfn}[Arithmetic-based Smullyan frames]\leavevmode
\begin{itemize}
    \item $\Sigma_{\R} : = \{a_i \mid i \in \Gamma\} \cup \{\R\}$.
    \item $\Pred_{\R} = \{\R^j a_i \mid j \in \mathbb{N}$ and $i \in \Gamma\}$. 
    \item $\Sigma_{\N \R} : = \{a_i \mid i \in \Gamma\}\cup \{\N, \R\}$.
    \item $\Pred_{\N \R} : = \{X a_i \mid X \in \{\N, \R\}^\ast$ and $i \in \Gamma\}$. 
\end{itemize}
\begin{enumerate}
    \item The tuple $F_{\R} : = (\Sigma_{\R}, \Pred_{\R})$ is called the \textit{arithmetic-based $\R$-Smullyan frame}. 
    \item The tuple $F_{\N \R} : = (\Sigma_{\N \R}, \Pred_{\N \R})$ is called the \textit{arithmetic-based $\N \R$-Smullyan frame}. 
\end{enumerate}    
\end{dfn}

In the following, $F_A$ is assumed to denote $F_{\R}$ or $F_{\N \R}$. 
Here, the subscript $A$ stands for `arithmetic'. 
If $F_A$ is $F_\R$, then $\Sigma_{F_A}$ and $\Pred_{F_A}$ denote $\Sigma_{\R}$ and $\Pred_{\R}$, respectively. 
Similarly, if $F_A$ is $F_{\N \R}$, then $\Sigma_{F_A}$ and $\Pred_{F_A}$ denote $\Sigma_{\N \R}$ and $\Pred_{\N \R}$, respectively. 
The following proposition easily follows from the definitions of $\Pred_{\R}$ and $\Pred_{\N \R}$. 

\begin{prop}\label{AB-models}
For any function $\Phi : \Pred_{F_A} \rightarrow \mathcal{P}(\Sigma_{F_A}^\ast)$, we have that $(F_A, \Phi) = (\Sigma_{F_A}, \Pred_{F_A}, \Phi)$ is a Smullyan model. 
\end{prop}

Notice that for each Smullyan model $M$, the set of all $M$-sentences is determined only by its frame $(\Sigma_M, \Pred_M)$. 
So, we write the set of all sentences determined by the frame $F_A$ as $\Stc_{F_A}$. 

    

To associate our arithmetic-based Smullyan frames with arithmetic, we introduce two mappings $I_{F_A}$ and $J_{F_A}$.
First, we introduce the mapping $I_{F_A}$ which assigns a $v$-formula to every $F_A$-predicate. 

\begin{dfn}[The mapping $I_{F_A}$]\label{def:I}
The mapping $I_{F_A}$ from $\Pred_{F_A}$ to the set of all $v$-formulas is defined as follows:
\begin{align*}
    I_{F_A}(X a_i) : \equiv 
    \begin{cases}
        \neg^k \, \varphi_i(v) & \text{if}\ ^\exists k \geq 0\, (X \equiv \N^k), \\
        \neg^{k+l} \, \varphi_i(d(v)) & \text{if}\ ^\exists k, l \geq 0\, (X \equiv \N^k \R \N^l), \\
        \neg^{k+l} \, \bot & \text{if}\ ^\exists k, l \geq 0\, ^\exists Y \, (X \equiv \N^k \R \N^l \R Y). 
    \end{cases}
\end{align*}
Here $\neg^k$ is the abbreviation for $\underbrace{\neg \cdots \neg}_{k}$. 
In particular, the definition of the mapping $I_{F_\R}$ is simply rewritten as follows: 
\begin{align*}
    I_{F_{\R}}(X a_i) : \equiv 
    \begin{cases}
        \varphi_i(v) & \text{if}\ X \equiv \epsilon, \\
        \varphi_i(d(v)) & \text{if}\ X \equiv \R, \\
        \bot & \text{if}\ ^\exists j \geq 2\, (X \equiv \R^j). 
    \end{cases}
\end{align*}
\end{dfn}


Next, we define the mapping $J_{F_A}$ which assigns an $\mathcal{L}_A$-sentence to each $F_A$-sentence. 

\begin{dfn}[The mapping $J_{F_A}$]\label{def:J}
The mapping $J_{F_A}$ from $\Stc_{F_A}$ to a set of $\mathcal{L}_A$-sentences is inductively defined as follows: 
Let $H \in \Pred_{F_A}$ and $X \in \Sigma_{F_A}^\ast$. 

\begin{enumerate}
 \item For $X \in \Pred_{F_A}$, $J_{F_A}(H X) : \equiv I_{F_A}(H)(\gn{I_{F_A}(X)(v)})$. 
 
 \item For $X \notin \Stc_{F_A}$,  
 \begin{itemize}
    \item $J_{F_A}(a_i X) : \equiv I_{F_A}(a_i)(0)$,
    \item $J_{F_A}(\R H X) : \equiv \bot$,
    \item $J_{F_A}(\N H X) : \equiv \lnot J_{F_A}(H X)$.
\end{itemize}

 
 \item For $X \in \Stc_{F_A}^+$, 
 \begin{itemize}
    \item $J_{F_A}(a_i X) : \equiv I_{F_A}(a_i)(\gn{J_{F_A}(X)})$,
    \item $J_{F_A}(\R H X) : \equiv \bot$, 
    \item $J_{F_A}(\N H X) : \equiv \lnot J_{F_A}(H X)$. 
\end{itemize}
In each of the cases of $X \notin \Stc_{F_A}$ and $X \in \Stc_{F_A}^+$, the third bullet point is applied only when $F_A$ is $F_{\N \R}$.
\end{enumerate}
\end{dfn}

For the arithmetic-based $\N \R$-Smullyan frame $F_{\N \R}$, we show that the symbol $\N$ plays the role of $\neg$ in arithmetic through the mappings $I_{F_{\N \R}}$ and $J_{F_{\N \R}}$.

\begin{prop}\label{ab_N}
Suppose $F_A = F_{\N \R}$. 
\begin{enumerate}
    \item For any $H \in \Pred_{F_A}$, we have $I_{F_A}(\N H) \equiv \neg I_{F_A}(H)$.
    \item For any $S \in \Stc_{F_A}$, we have $J_{F_A}(\N S) \equiv \lnot J_{F_A}(S)$. 
\end{enumerate}
\end{prop}
\begin{proof}
1. Since $H$ and $\N H$ contain the same number of $\R$'s, we have that $I_{F_A}(\N H)$ is exactly $\neg I_{F_A}(H)$ by the definition of $I_{F_A}$. 

2. Let $S \equiv HX$ for $H \in \Pred_{F_A}$ and $X \in \Sigma_{F_A}^\ast$. 
If $X \notin \Pred_{F_A}$, then $X \notin \Stc_{F_A}$ or $X \in \Stc_{F_A}^+$, and hence this proposition is trivial by the definition of $J_{F_A}$. 
If $X \in \Pred_{F_A}$, then we have
\begin{align*}
    J_{F_A}(\N H X) & \overset{\textrm{Def.}~\ref{def:J}}{\equiv} I_{F_A}(\N H)(\gn{I_{F_A}(X)(v)}) \\
    & \overset{\textrm{Clause 1}}{\equiv} \neg I_{F_A}(H)(\gn{I_{F_A}(X)(v)}) \\
    & \overset{\textrm{Def.}~\ref{def:J}}{\equiv} \neg J_{F_A}(H X). \qedhere
\end{align*}
\end{proof}


Next, we investigate the effect of the symbol $\R$ in arithmetic through the mapping $J_{F_{\N \R}}$.
For this purpose, we prepare the following lemma. 

\begin{lem}\label{ab_lem}
Let $H \in \Pred_{F_A}$ and $X \in \Stc_{F_A}^+$. 
If $H$ contains no $\R$, then $J_{F_A}(H X) \equiv I_{F_A}(H)(\gn{J_{F_A}(X)})$. 
\end{lem}
\begin{proof}
    It suffices to show that for every $k \geq 0$ and $i \in \Gamma$, we have $J_{F_A}(\N^k a_i X) \equiv I_{F_A}(\N^k a_i)(\gn{J_{F_A}(X)})$. 
    We prove this by induction on $k$. 

For $k = 0$, $J_{F_A}(a_i X) \equiv I_{F_A}(a_i) (\gn{J_{F_A}(X)})$ directly follows from Definition \ref{def:J}. 

Suppose that the lemma holds for $k$. 
We prove the lemma holds for $k+1$. 
In this case, we have $F_A = F_{\N \R}$. 
    \begin{align*}
J_{F_A}(\N^{k+1} a_i X) & \overset{\textrm{Prop.}~\ref{ab_N}.(2)}{\equiv} \neg J_{F_A}(\N^k a_i X) \\
    & \overset{\textrm{I.H.}}{\equiv} \neg I_{F_A}(\N^k a_i)(\gn{J_{F_A}(X)}) \\
    & \overset{\textrm{Prop.}~\ref{ab_N}.(1)}{\equiv} I_{F_A}(\N^{k+1} a_i)(\gn{J_{F_A}(X)}). \qedhere
    \end{align*}
\end{proof}

\begin{prop}\label{ab_R}
Let $H \in \Pred_{F_A}$ and $K \in \Sigma_{F_A}^\ast$. 
\begin{enumerate}
    \item If $K \notin \Pred_{F_A}$, then $J_{F_A}(\R H K) \equiv \bot$. 
    
    \item If $K \in \Pred_{F_A}$, then $\PA \vdash J_{F_A}(\R H K) \leftrightarrow J_{F_A}(H K K)$. 
\end{enumerate}
\end{prop}
\begin{proof}
1. Let $K \notin \Pred_{F_A}$. 
By Definition \ref{def:J}, we have $J_{F_A}(\R H K) \equiv \bot$.

2. Let $K \in \Pred_{F_A}$. 
We distinguish the following two cases: 

\medskip

Case 1: $H$ contains at least one $\R$. \\
In this case, we find some $l \geq 0$ and $H' \in \Pred_{F_A}$ such that $H \equiv \N^l \R H'$. 
\begin{align*}
J_{F_A}(\R \N^l \R H' K) & \overset{\textrm{Def.}~\ref{def:J}}{\equiv} I_{F_A}(\R \N^l \R H')(\gn{J_{F_A}(K)(v)}) \\
 & \overset{\textrm{Def.}~\ref{def:I}}{\equiv} \neg^l\, \bot \\
 & \overset{\textrm{Def.}~\ref{def:J}}{\equiv} \neg^l\, J_{F_A}(\R H' K K) \tag{$KK \in \Stc_{F_A}^+$} \\
 & \overset{\textrm{Prop.}~\ref{ab_N}.(2)}{\equiv} J_{F_A}(\N^l \R H' KK). 
\end{align*}

\medskip

Case 2: $H$ contains no $\R$. 
\begin{align*}
J_{F_A}(\R H K) & \overset{\textrm{Def.}~\ref{def:J}}{\equiv} I_{F_A}(\R H)(\gn{I_{F_A}(K)(v)}) \\
& \overset{\textrm{Def.}~\ref{def:I}}{\equiv} I_{F_A}(H)(d(\gn{I_{F_A}(K)(v)})) \\
& \overset{\textrm{Def.~of}\ d}{\leftrightarrow_{\PA}} I_{F_A}(H)(\gn{I_{F_A}(K)(\gn{I_{F_A}(K)(v)})}) \\
& \overset{\textrm{Def.}~\ref{def:J}}{\equiv}  I_{F_A}(H)(\gn{J_{F_A}(K K)}) \\
& \overset{\textrm{Lem.}~\ref{ab_lem}}{\equiv}  J_{F_A}(H K K). \tag{$KK \in \Stc_{F_A}^+$} 
\end{align*}
\end{proof}



\subsection{The $\mathbb{N}$-based Smullyan model $M_{\mathbb{N}}$}

In this subsection, we introduce the Smullyan model $M_{\mathbb{N}}$ which is defined by referring to the standard model $\mathbb{N}$ of arithmetic. 
We prove that $\mathbb{N}$ is actually an $\N \R$-Smullyan model. 
Then, we prove that $\FPT$, $\STTarski$, and $\SNGodel$ for $M_{\mathbb{N}}$ correspond to the Fixed Point Theorem over $\mathbb{N}$, the original Tarski's Undefinability Theorem, and a version of G\"odel's First Incompleteness Theorem, respectively. 

\begin{dfn}[The Smullyan model $M_{\mathbb{N}}$]\leavevmode
\begin{enumerate}
    \item We define the function $\Phi_{\mathbb{N}}: \Pred_{F_{\N \R}} \rightarrow \mathcal{P}(\Sigma_{F_{\N \R}}^\ast)$ as follows: 
    \[
        \Phi_{\mathbb{N}}(H) : = \{X \in \Sigma_{F_{\N \R}}^\ast \mid \mathbb{N} \models J_{F_{\N \R}}(HX)\}.
    \]

    \item The triple $M_{\mathbb{N}} : = (F_{\N \R}, \Phi_{\mathbb{N}}) = (\Sigma_{\N \R}, \Pred_{\N \R}, \Phi_{\mathbb{N}})$ is called the \textit{$\mathbb{N}$-based Smullyan model}. 
\end{enumerate}
\end{dfn}

From the definition of $\Phi_{\mathbb{N}}$, we obtain
\begin{equation}\label{True_N}
\True_{M_{\mathbb{N}}} = \{S \in \Stc_{F_{\N \R}} \mid \mathbb{N} \models J_{F_{\N \R}}(S)\}.
\end{equation}

\begin{thm}
The $\mathbb{N}$-based Smullyan model $M_{\mathbb{N}}$ is an $\N\R$-Smullyan model.
\label{THM_Nbased}
\end{thm}
\begin{proof}
We mentioned in Proposition \ref{AB-models} that $M_{\mathbb{N}}$ is a Smullyan model. 
So, it suffices to show that $\Phi_{\mathbb{N}}$ satisfies the requirements concerning $\N$ and $\R$.
Let $H \in F_{\N \R}$.  
    \begin{align*}
        \Phi_{\mathbb{N}}(\N H)  & = \{X \in \Sigma_{F_{\N \R}}^\ast  \mid \mathbb{N} \models J_{F_{\N \R}}(\N H X)  \} \\
         & = \{X \in \Sigma_{F_{\N \R}}^\ast \mid \mathbb{N} \models \lnot J_{F_{\N \R}}(H X)  \} \tag{by Proposition \ref{ab_N}.(2)}\\
         & = \Sigma^\ast_{F_{\N \R}} \setminus \{X \in \Sigma^\ast_{F_{\N \R}} \mid \mathbb{N} \models  J_{F_{\N \R}}(HX)  \}  \\
         & = \Sigma^\ast_{F_{\N \R}} \setminus \Phi_{\mathbb{N}}(H). 
    \end{align*}

    \begin{align*}
    \Phi_{\mathbb{N}}(\R H) & = \{X \in \Sigma^\ast_{F_{\N \R}} \mid \mathbb{N} \models J_{F_{\N \R}}(\R HX)\}\\
    & =  \{K \in \Pred_{F_{\N \R}} \mid \mathbb{N} \models J_{F_{\N \R}}(H K K) \} \tag{by Proposition \ref{ab_R}}\\
    & = \{K \in \Pred_{F_{\N \R}} \mid KK \in \Phi_{\mathbb{N}}(H)\}. \qedhere 
    \end{align*}
\end{proof}

Therefore, by the results we have obtained so far, $M_{\mathbb{N}}$ satisfies $\FPT$, $\STTarski$, and $\SNGodel$. 
We show that each of these facts yields a meaningful property in arithmetic. 
First, we show that the Fixed Point Theorem over $\mathbb{N}$ follows from $\FPT$ for $M_{\mathbb{N}}$. 

\begin{thm}[Fixed Point Theorem over $\mathbb{N}$]\label{M_N:FPT}
The following statement follows from $\FPT$ for $M_{\mathbb{N}}$: ``For any $v$-formula $\varphi_i(v)$, there exists an $\mathcal{L}_A$-sentence $\theta$ such that $\mathbb{N} \models \theta \leftrightarrow \varphi_i(\gn{\theta})$''. 
\end{thm}
\begin{proof}
Let $\varphi_i(v)$ be any $v$-formula. 
By $\FPT$ for $M_{\mathbb{N}}$, we find an $M_{\mathbb{N}}$-fixed point $S \in \Stc_{F_{\N \R}}^+$ of $a_i \in \Pred_{F_{\N \R}}$.
The following equivalences show that the $\mathcal{L}_A$-sentence $J_{F_{\N \R}}(S)$ is a fixed point of $\varphi_i(v)$ over $\mathbb{N}$. 
\begin{align*}
    \mathbb{N} \models J_{F_{\N \R}}(S) & \overset{(\ref{True_N})}{\iff} S \in \True_{M_{\mathbb{N}}} \iff M_{\mathbb{N}} \models S \\
    & \overset{\FPT}{\iff} M_{\mathbb{N}} \models a_i S \iff a_i S \in \True_{M_{\mathbb{N}}} \\
    & \overset{(\ref{True_N})}{\iff} \mathbb{N} \models J_{F_{\N \R}}(a_i S) \\
    & \overset{\textrm{Lem.}~\ref{ab_lem}}{\iff} \mathbb{N} \models I_{F_{\N \R}}(a_i)(\gn{J_{F_{\N \R}}(S)}) \tag{$S \in \Stc_{F_{\N \R}}^+$}\\
    & \overset{\textrm{Def.}~\ref{def:I}}{\iff} \mathbb{N} \models \varphi_i(\gn{J_{F_{\N \R}}(S)}).
 \qedhere\end{align*}
\end{proof}

Second, we show that $\STTarski$ for $M_{\mathbb{N}}$ corresponds to the original version of Tarski's Undefinability Theorem. 

\begin{thm}[Tarski's Undefinability Theorem]\label{M_N:Tarski}
The following statement follows from $\STTarski$ for $M_{\mathbb{N}}$: ``The set $\mathsf{TA}$ of all $\mathcal{L}_A$-sentences true in $\mathbb{N}$ is not definable in $\mathbb{N}$''. 
\end{thm}
\begin{proof}
Let $\varphi_i(v)$ be any $v$-formula. 
By $\STTarski$ for $M_{\mathbb{N}}$, we have that $\True_{M_{\mathbb{N}}}^+ \neq \Phi_{\mathbb{N}}(a_i) \cap \Stc_{F_{\N \R}}^+$.  
That is, there exists $S \in \Stc_{F_{\N \R}}^+$ such that $S \in \True_{M_{\mathbb{N}}} \iff S \notin \Phi_{\mathbb{N}}(a_i)$. 
Then, 
\begin{align*}
    \mathbb{N} \models J_{F_{\N \R}}(S) & \overset{(\ref{True_N})}{\iff} S \in \True_{M_{\mathbb{N}}} \iff S \notin \Phi_{\mathbb{N}}(a_i) \iff a_i S \notin \True_{M_{\mathbb{N}}} \\
    & \overset{(\ref{True_N})}{\iff} \mathbb{N} \not \models J_{F_{\N \R}}(a_i S) \overset{\textrm{Lem.}~\ref{ab_lem}}{\iff} \mathbb{N} \not \models I_{F_{\N \R}}(a_i)(\gn{J_{F_{\N \R}}(S)}) \tag{$S \in \Stc_{F_{\N \R}}^+$}\\
    & \overset{\textrm{Def.}~\ref{def:I}}{\iff} \mathbb{N} \not \models \varphi_i(\gn{J_{F_{\N \R}}(S)}).
\end{align*}
These equivalences show that $\varphi_i(v)$ does not define $\mathsf{TA}$. 
\end{proof}


Finally, we show that $\SNGodel$ for $M_{\mathbb{N}}$ corresponds to a version of G\"odel's First Incompleteness Theorem with respect to arithmetically definable sound theories (cf.~Kikuchi and Kurahashi \cite{KK17} and Salehi and Seraji \cite{SS17}). 

\begin{thm}[A version of G\"odel's First Incompleteness Theorem]\label{M_N:Godel}
The following statement follows from $\SNGodel$ for $M_{\mathbb{N}}$: ``Every arithmetically definable sound $\mathcal{L}_A$-theory is incomplete''. 
\end{thm}
\begin{proof}
Let $T$ be any arithmetically definable sound $\mathcal{L}_A$-theory. 
We find an $\mathcal{L}_A$-formula $\tau(x)$ defining (the set of all G\"odel numbers of) $T$ in $\mathbb{N}$. 
As in the usual proof of G\"odel's incompleteness theorems, we can construct a provability predicate $\Pr_T(v)$ of $T$ by using $\tau(v)$, that is, for any $\mathcal{L}_A$-formula $\psi$, we have $T \vdash \psi \iff \mathbb{N} \models \Pr_T(\gn{\psi})$. 
Let $i \in \Gamma$ be such that $\varphi_i(v) \equiv \Pr_T(v)$. 

We shall show $\Phi_{\mathbb{N}}(a_i) \cap \Stc_{F_{\N \R}}^+ \subseteq \True_{M_{\mathbb{N}}}^+$. 
Let $S_0 \in \Phi_{\mathbb{N}}(a_i) \cap \Stc_{F_{\N \R}}^+$. 
Then, $a_i S_0 \in \True_{M_{\mathbb{N}}}$. 
As in the proof of Theorem \ref{M_N:FPT}, we have $\mathbb{N} \models \varphi_i(\gn{J_{F_{\N \R}}(S_0)})$. 
Hence, $T \vdash J_{F_{\N \R}}(S_0)$ because $\varphi_i(v)$ is identical to $\Pr_T(v)$. 
Since $T$ is sound, we have $\mathbb{N} \models J_{F_{\N \R}}(S_0)$. 
Thus, by (\ref{True_N}), we conclude $S_0 \in \True_{M_{\mathbb{N}}}^+$. 

By $\SNGodel$ for $M_{\mathbb{N}}$, we get $S \in \Stc_{F_{\N \R}}^+$ such that $S \notin \Phi_{\mathbb{N}}(a_i)$ and $\N S \notin \Phi_{\mathbb{N}}(a_i)$. 
Since $a_i S$ and $a_i \N S$ are not in $\True_{M_{\mathbb{N}}}$, we have $\mathbb{N} \not \models \varphi_i(\gn{J_{F_{\N \R}}(S_0)})$ and $\mathbb{N} \not \models \varphi_i(\gn{J_{F_{\N \R}}(\N S_0)})$. 
By Proposition \ref{ab_N}, we also have $\mathbb{N} \not \models \varphi_i(\gn{\neg J_{F_{\N \R}}(S_0)})$. 
We conclude that $T \nvdash J_{F_{\N \R}}(S_0)$ and $T \nvdash \neg J_{F_{\N \R}}(S_0)$. 
\end{proof}

\subsection{$\PA$-based Smullyan model $M_{\PA}$}

In this subsection, we introduce the Smullyan model $M_{\PA}$ which is defined by referring to Peano arithmetic $\PA$. 
We show that $M_{\PA}$ is a witness of the non-implication $\SFTarski \ \&\ \TTarski \overset{\R}{\not \Rightarrow} \STTarski$ (cf.~Proposition \ref{SFTTT-non-STT}). 

\begin{dfn}[The Smullyan model $M_{\PA}$]\leavevmode
    \begin{enumerate}
        \item We define the function $\Phi_{\PA}: \Pred_{F_{\R}} \rightarrow \mathcal{P}(\Sigma_{F_{\R}}^\ast)$ as follows: 
        \[
            \Phi_{\PA}(H) : = \{X \in\ \Sigma_{F_{\R}}^\ast \mid \PA \vdash J_{F_{\R}}(H X) \}
        \]

        \item The triple $M_{\PA} := (F_{\R} , \Phi_{\PA}) = (\Sigma_{F_\R}, \Pred_{F_\R}, \Phi_{\PA})$ is called the \textit{$\PA$-based Smullyan model}. 
    \end{enumerate}
\end{dfn}

From the definition of $\Phi_\PA$, we obtain
\begin{equation}\label{True_PA}
\True_{M_\PA} = \{S \in \Stc_{F_{\R}} \mid \PA \vdash J_{F_{\R}}(S)\}.
\end{equation}

The following theorem is proved as in the $\R$-part of the proof of Theorem \ref{THM_Nbased}. 

\begin{thm}\label{M_PA:RMODEL}
The $\PA$-based Smullyan model $M_{\PA}$ is an $\R$-Smullyan model.
\end{thm}



    


So, by Theorem \ref{THM_FPT}, $M_{\PA}$ satisfies $\FPT$. 
Also by Proposition \ref{FPT-SFT}, $M_\PA$ satisfies $\SFTarski$. 
The following theorem stating that $\FPT$ for $M_\PA$ corresponds to a weak version of the Fixed Point Theorem over $\PA$ is also proved as in the proof of Theorem \ref{M_N:FPT}. 

\begin{thm}[A weak version of the Fixed Point Theorem over $\PA$]\label{M_PA:FPT}
The following statement follows from $\FPT$ for $M_\PA$: ``For every $v$-formula $\varphi_i(v)$, there exists an $\mathcal{L}_A$-sentence $\theta$ such that $\PA \vdash \theta \iff \PA \vdash \varphi_i(\gn{\theta})$''. 
\end{thm}

The Fixed Point Theorem over $\PA$ used in the classic proof of the incompleteness theorems is of the form $\PA \vdash \theta \leftrightarrow \varphi_i(\gn{\theta})$, and the statement above is indeed weak. 
However, this weak version of the Fixed Point Theorem is sufficient for a proof of the incompleteness of $\PA$ as follows: 
Let $\varphi_i(v)$ be a $\Sigma_1$ provability predicate of $\PA$, that is, for any $\mathcal{L}_A$-formula $\psi$, we have $\PA \vdash \psi \iff \PA \vdash \varphi_i(\gn{\psi})$. 
By Theorem \ref{M_PA:FPT}, there exists an $\mathcal{L}_A$-sentence $\theta$ such that $\PA \vdash \theta \iff \PA \vdash \neg \varphi_i(\gn{\theta})$. 
If $\PA \vdash \theta$, then $\PA \vdash \neg \varphi_i(\gn{\theta})$ and $\PA \vdash \varphi_i(\gn{\theta})$. 
This contradicts the consistency of $\PA$. 
Hence, we have $\PA \nvdash \theta$. 
It follows that $\PA \nvdash \varphi_i(\gn{\theta})$ and $\PA \nvdash \neg \varphi_i(\gn{\theta})$.  
Therefore, $\PA$ is incomplete. 
This weak version of the Fixed Point Theorem was discussed in Moschovakis \cite[Theorem 5.1]{Mos10} (see also Salehi \cite[Remark 3.5]{sal22}).

Finally, we prove that the $\PA$-based Smullyan model $M_{\PA}$ is also a witness of the non-implication $\SFTarski \ \&\ \TTarski \overset{\R}{\not \Rightarrow} \STTarski$. 
We already mentioned that $M_{\PA}$ satisfies $\SFTarski$.

\begin{thm}\label{M_PA:non-TT}\leavevmode
\begin{enumerate}
    \item $M_\PA$ does not satisfy $\TTarski^+$. 
    \item $M_\PA$ satisfies $\TTarski$. 
\end{enumerate}
\end{thm}
\begin{proof}
1. Let $\varphi_i(v)$ be a provability predicate of $\PA$. 
We shall prove $\True_{M_{\PA}}^+ = \Phi_{\PA}(a_i) \cap \Stc_{F_{\R}}^+$, which shows that $M_{\PA}$ does not satisfy $\STTarski$. 
Let $S \in \Stc_{F_{\R}}^+$. 
\begin{align*}
    S \in \True_{M_{\PA}} & \iff \PA \vdash J_{F_\R}(S) \iff \PA \vdash \varphi_i(\gn{J_{F_\R}(S)}) \\
    & \iff \PA \vdash J_{F_\R}(a_i S) \iff S \in \Phi_{\PA}(a_i). 
\end{align*}

2. Suppose, towards a contradiction, that $H \in \Pred_{F_\R}$ names $\True_{M_{\PA}}$. 
Let $\varphi_i$ be an $\mathcal{L}_A$-sentence such that $\PA \vdash \varphi_i$. 
Then, we have $I_{F_\R}(a_i) \equiv J_{F_\R}(a_i) \equiv I_{F_\R}(r a_i) \equiv \varphi_i$ because $\varphi_i$ is a sentence. 
On the other hand, $J_{F_\R}(r a_i) \equiv J_{F_\R}(r a_i \epsilon) \equiv \bot$. 

Since $\PA \vdash \varphi_i$, we get $\PA \vdash J_{F_\R}(a_i)$, and hence $a_i \in \True_{M_{\PA}}$. 
By the supposition, we have $a_i \in \Phi_{\PA}(H)$, and hence $\PA \vdash J_{F_\R}(Ha_i)$. 
Since $a_i \in \Pred_{F_\R}$, we have $\PA \vdash J_{F_\R}(H)(\gn{I_{F_\R}(a_i)})$ and so $\PA \vdash J_{F_\R}(H)(\gn{\varphi_i})$. 
It follows that $\PA \vdash J_{F_\R}(H)(\gn{I_{F_\R}(r a_i)})$. 
Then, $\PA \vdash J_{F_\R}(H r a_i)$ and thus $r a_i \in \Phi_{\PA}(H)$. 
By the supposition, we have $r a_i \in \True_{M_{\PA}}$. 
Therefore, $\PA \vdash J_{F_\R}(r a_i)$ and this means $\PA \vdash \bot$. 
This is a contradiction. 
\end{proof}

\section{Concluding remarks}

In this paper, we focused on Smullyan's paper ``Truth and Provability'' (2013, \textit{The Mathematical Intelligencer}). 
We introduced the notion of Smullyan models in Section \ref{SmullyanModels}. 
In order to understand deeply the content of Smullyan's paper, we have mainly studied this notion throughout the present paper. 
In Section \ref{properties}, we introduced several properties of Smullyan models and studied the equivalences between some of these properties. 
These equivalences are summarized in Figure \ref{Fig1}. 
In particular, these equivalences show that it is sufficient to consider the four properties $\TTarski$, $\FTarski$, $\STTarski$, and $\SFTarski$ when dealing with the properties we introduced.
In Section \ref{non-implications}, we provided several Smullyan models which are witnesses of non-implications between some combinations of these four properties. 
As a consequence of these non-implications, no more arrows can be added into Figure \ref{Fig1}. 
Finally in Section \ref{arithmetic}, we introduced the $\N \R$-Smullyan model $M_{\mathbb{N}}$ and the $\R$-Smullyan model $M_{\PA}$ based on $\mathbb{N}$ and $\PA$, respectively. 
We proved that $\STTarski$ for $M_{\mathbb{N}}$ corresponds to the original Tarski's Undefinability Theorem. 
Also, we showed that $M_{\PA}$ is a witness of the non-implication $\SFTarski \ \&\ \TTarski \overset{\R}{\not \Rightarrow} \STTarski$. 

Through these studies, we have found that Smullyan's framework is simple but powerful, and has a fertile structure. 
In particular, it may be said that the essence of the proof of Tarski's Undefinability Theorem is reasonably explained within this framework as indicated in Theorem \ref{M_N:Tarski}. 
However, this framework seems too simple for a more intricate discussion of the First Incompleteness Theorem. 
In his 2013 book \cite{Smu13b}, Smullyan introduced an extended framework that enables the analysis of Rosser's First Incompleteness Theorem.
It will be interesting to explore what extensions, including this one, can make sense for the analysis of the First Incompleteness Theorem in a future work.

\section*{Acknowledgements}

We would like to thank Naoyuki Hatanaka and Fumihisa Ikema for their comments during the early stages of the study. 
We would also like to thank the anonymous referee for suggesting some improvements.
The first author was supported by JSPS KAKENHI Grant Number JP23K03200.

\bibliographystyle{plain}
\bibliography{ref}

\end{document}